\documentclass[a4, 12pt, 
]{amsart}
\usepackage[utf8]{inputenc}
\usepackage[T1]{fontenc}

\usepackage[colorlinks=true, linkcolor=blue, urlcolor=magenda, citecolor=red]{hyperref}

\usepackage{color}
\usepackage{xcolor} 
\usepackage{ulem}
\usepackage{calligra}
\usepackage{mathrsfs}

\usepackage[all]{xy}

\newtheorem{theo}{Theorem}[section]

\newtheorem{lemm}[theo]{Lemma}
\newtheorem{cor}[theo]{Corollary}

\numberwithin{equation}{section}

\theoremstyle{definition}

\theoremstyle{remark}
\newtheorem{rem}[theo]{Remark}

\newcommand{\Alb}[0]{\operatorname{Alb}}

\newcommand{\Image}[0]{\operatorname{Im}\hspace{-0.03cm}}

\newcommand{\codim}[0]{\operatorname{codim}}
\newcommand{\Sym}[0]{\operatorname{Sym}}
\newcommand{\GL}[0]{\operatorname{GL}}

\newcommand{\sat}{{\rm{sat}}}

\newcommand{\cal}[1]{\mathcal{#1}}

\makeatletter
\newcommand*{\da@rightarrow}{\mathchar"0\hexnumber@\symAMSa 4B }
\newcommand*{\da@leftarrow}{\mathchar"0\hexnumber@\symAMSa 4C }
\newcommand*{\xdashrightarrow}[2][]{%
  \mathrel{%
    \mathpalette{\da@xarrow{#1}{#2}{}\da@rightarrow{\,}{}}{}%
  }%
}
\newcommand{\xdashleftarrow}[2][]{%
  \mathrel{%
    \mathpalette{\da@xarrow{#1}{#2}\da@leftarrow{}{}{\,}}{}%
  }%
}
\newcommand*{\da@xarrow}[7]{%
  \sbox0{$\ifx#7\scriptstyle\scriptscriptstyle\else\scriptstyle\fi#5#1#6\m@th$}%
  \sbox2{$\ifx#7\scriptstyle\scriptscriptstyle\else\scriptstyle\fi#5#2#6\m@th$}%
  \sbox4{$#7\dabar@\m@th$}%
  \dimen@=\wd0 %
  \ifdim\wd2 >\dimen@
    \dimen@=\wd2 %
  \fi
  \count@=2 %
  \def\da@bars{\dabar@\dabar@}%
  \@whiledim\count@\wd4<\dimen@\do{%
    \advance\count@\@ne
    \expandafter\def\expandafter\da@bars\expandafter{%
      \da@bars
      \dabar@ 
    }%
  }%
  \mathrel{#3}%
  \mathrel{%
    \mathop{\da@bars}\limits
    \ifx\\#1\\%
    \else
      _{\copy0}%
    \fi
    \ifx\\#2\\%
    \else
      ^{\copy2}%
    \fi
  }%
  \mathrel{#4}%
}
\makeatother

\usepackage{geometry}
\begin{document}

\title[Compact K\"ahler manifolds with psef tangent bundle]
{On compact K\"ahler manifolds with \\ pseudo-effective tangent bundle}

\author{Shin-ichi MATSUMURA}

\address{Mathematical Institute, Tohoku University, 
$\&$ Division for the Establishment of Frontier Science of Organization for Advanced Studies, 
6-3, Aramaki Aza-Aoba, Aoba-ku, Sendai 980-8578, Japan.}

\email{{\tt mshinichi-math@tohoku.ac.jp}}
\email{{\tt mshinichi0@gmail.com}}

\author{Chenghao Qing}
\address{Yau Mathematical Sciences Center, Tsinghua University, 30 Shuangqing Road, Haidian District, Beijing 100084, China}
\email{\tt qingchenghao@amss.ac.cn}
\date{\today, version 0.01}

\renewcommand{\subjclassname}{%
\textup{2020} Mathematics Subject Classification}


\maketitle

\begin{abstract}
In this paper, we prove that a compact K\"ahler manifold $X$ with pseudo-effective (resp.\,singular positively curved) tangent bundle admits a smooth (resp.\,locally constant) rationally connected fibration $\phi \colon X \to Y$ onto a finite \'etale quotient $Y$ of a compact complex torus. This result extends the structure theorem previously established for smooth projective varieties to compact K\"ahler manifolds.
\end{abstract}

\section{Introduction}\label{Sec1}

The pioneering works of \cite{HSW81, MZ86, Mok88} and \cite{CP91, DPS94} 
revealed the structure of a compact K\"ahler manifold $X$ whose tangent bundle $T_{X}$ 
either has semi-positive holomorphic bisectional curvature or is nef, 
leading to the decomposition of $X$ into a Fano manifold and a (compact complex) torus 
via the naturally associated fibration. 
In \cite{HIM22}, we attempted to generalize these results to the setting of ``singular Hermitian metrics'', 
using the notions of singular positively curved and pseudo-effective tangent bundles. 
Consequently, the following structure theorem can be proved when $X$ is a smooth \textit{projective} variety 
(see \cite{IMZ} for a generalization to klt varieties). 
Although the extension of this result to compact K\"ahler manifolds was naturally expected, 
it remained unsolved, as posed in \cite[Problem 4.5]{Mat22b}. 
In this paper, we establish the structure theorem for compact K\"ahler manifolds, thus completing this extension.

\begin{theo}\label{thm-main}
Let $X$ be a compact K\"ahler manifold with pseudo-effective tangent bundle 
$($see Subsection \ref{subsec-notation} for our definition of pseudo-effective vector bundles$)$. 
Then $X$ admits a fibration $\phi: X \to Y$ $($i.e.,\,a surjective holomorphic map with connected fibers$)$
with the following properties\,$:$
\begin{itemize}
\item[(1)] The fibration $\phi: X \to Y$ is smooth $($i.e.,\,all fibers are smooth$)$. 
\item[(2)] The base $Y$ is a finite \'etale quotient of a $($compact complex$)$ torus 
$($i.e.,\,there exists a finite \'etale cover $T\to Y$ by a torus $T$$)$. 
\item[(3)] A very general fiber $F$ of $\phi$ is rationally connected. 
\item[(4)] A very general fiber $F$ of $\phi$  has the pseudo-effective tangent bundle. 
\end{itemize}
Moreover, if we further assume that 
$T_X$ admits a $($possibly singular$)$ positively curved singular Hermitian metric, 
then we have$:$
\begin{itemize}
\item[(5)] The morphism $\phi: X \to Y$ is a locally constant fibration 
$($see \cite[Definition 2.3]{MW} for the definition$).$ In particular, 
it  is a locally trivial fibration.
\end{itemize} 
\end{theo}

\begin{rem}\label{rem-main}
In \cite{HIM22}, for a smooth projective variety $X$, 
we only proved that the fibration $\phi: X \to Y$ is locally trivial, which is weaker than the local constancy in (5). 
The local constancy in (5)  is proved by \cite[Lemma 3.1]{Mul} and \cite[Theorem 1.1 (5)]{HIM22}.
\end{rem}

\begin{cor}\label{cor-pi}
Let $X$ be a compact K\"ahler manifold with pseudo-effective tangent bundle. 
Then, the $($topological$)$ fundamental group $\pi_{1}(X)$ is virtually abelian 
$($i.e.,\,there exists an abelian subgroup of  $\pi_{1}(X)$ of finite index$).$
\end{cor}

At the end of this section, we outline the proof of the first conclusion of Theorem \ref{thm-main}. 
In \cite{HIM22}, for a smooth projective variety $X$, we consider an MRC fibration $X \dashrightarrow Y$ to obtain the desired fibration. 
A key point here is that the canonical bundle $K_{Y}$ of a smooth base $Y$ of an MRC fibration is pseudo-effective, 
which relies on the projectivity of $X$ (and also $Y$) when applying the result of \cite{BDPP13} together with \cite{GHS03}. 
(We note that, during the preparation of this paper, the result of \cite{BDPP13} was generalized to compact K\"ahler manifolds by \cite{Ou}, 
but we believe our methods offer a different perspective and thus deserve to be presented.)
In this paper, we consider the Albanese map $\alpha \colon X \to Y := \Alb(X)$ instead of MRC fibrations. 
The main difficulty in this setting is to prove that a very general fiber $F$ is rationally connected. 

The argument in \cite[Proposition 3.12]{DPS94} shows that we may assume the augmented irregularity $\hat{q}(F)$ vanishes 
after passing to a finite \'etale cover of $X$. 
Furthermore, \cite[Theorem 3.12]{HIM22} shows that the tangent bundle $T_{F}$ is also pseudo-effective. 
Once we know $F$ is projective, we can deduce that $F$ is rationally connected 
by applying the structure theorem for smooth projective varieties proved in \cite{HIM22}. 
Hence, the main difficulty reduces to proving that a compact K\"ahler manifold $X$ is projective 
provided that the tangent bundle $T_{X}$ is pseudo-effective and its augmented irregularity $\hat{q}(X)$ vanishes. 
To prove this, focusing on the notion of special varieties introduced by \cite{Cam04}, we adopt an idea from \cite{Mat}, 
which extends the result of \cite{Mat20, Mat22a} to compact K\"ahler manifolds. 
We first show that any representation $\rho \colon \pi_{1}(X) \to \GL(r\mathord{:}\,\mathbb{C})$ has a virtually abelian image 
(see Theorem \ref{thm-sp}). 
If $X$ were not projective, there would exist a non-trivial holomorphic $2$-form on $X$, 
implying that the cotangent bundle $\Omega_{X}$ admits a flat subvector bundle $\mathcal{F}$. 
Consider the representation $\rho \colon \pi_{1}(X) \to \GL(r\mathord{:}\,\mathbb{C})$ associated with $\mathcal{F}$. 
Then, by the assumption $\hat{q}(F)=0$ and the virtual abelianity, 
we conclude that $\mathcal{F}$ is \'etale trivializable (see Theorem \ref{thm-irr}), 
which contradicts the assumption $\hat{q}(F)=0$.
Hence, the fiber $F$ should be projective, and thus rationally connected.

\subsection*{Acknowledgments}\label{subsec-ack}
The first author would like to thank Dr.\,Niklas M\"uller 
for discussions related to \cite[Lemma 3.1]{Mul}. 
The first author was partially supported by Grant-in-Aid for Scientific Research (B) $\sharp$21H00976 from JSPS, 
Fostering Joint International Research (A) $\sharp$19KK0342 from JSPS, 
and the JST FOREST Program, $\sharp$PMJFR2368 from JST. 
The second author wishes to express his gratitude to Professor Juanyong Wang 
for suggesting that he consider the relevant issues and for kindly answering his questions in algebraic geometry. 
He would also like to thank Professor Xiangyu Zhou for his guidance and encouragement, 
and Professor Xiaokui Yang for his support and helpful advice.

\subsection{Notation and Conventions}\label{subsec-notation}

We use the terms ``invertible sheaves'' and ``line bundles'' interchangeably, and adopt the additive notation for tensor products (e.g.,\,$L+M:=L\otimes M$ for line bundles $L$ and $M$). Additionally, we use the terms ``locally free sheaves'' and ``vector bundles'' interchangeably. The term ``fibrations'' refers to a proper surjective morphism with connected fibers, and the term ``finite \'etale covers'' refers to an unramified finite surjective morphism.

Let us clarify our definition of pseudo-effectivity for a vector bundle $E$ on a compact complex manifold $X$. 
Fix a Hermitian form $\omega$ on $X$. 
We say that $E$ is \textit{pseudo-effective} if, for every $m \in \mathbb{Z}_{+}$, there exists a singular Hermitian metric $h_{m}$ on the $m$-th symmetric power 
$\Sym^{m} {E}$ such that 
$\log |u|_{h_{m}^{*}}$ is $\omega$-psh for any local section $u$ of the dual vector bundle $\Sym^m{E}^{*}$.


When $X$ is a smooth projective variety, 
this definition is equivalent to requiring that the non-nef locus of $\mathcal{O}_{\mathbb{P}(E)}(1)$ be non-dominant over $X$ 
(see \cite[Proposition 2.4]{Mat23}), 
where $\mathcal{O}_{\mathbb{P}(E)}(1)$ is the hyperplane bundle on the projective space bundle $\mathbb{P}(E) \to X$. 
This definition is stronger than merely requiring that $\mathcal{O}_{\mathbb{P}(E)}(1)$ be pseudo-effective, 
and is sometimes called ``strongly pseudo-effective'', 
but in this paper we simply refer to it as ``pseudo-effective''. 
See \cite{BKK}, \cite[Subsections 2.1, 2.2]{HIM22}, and \cite[Subsections 2.1, 2.2]{Mat23}
for further details on pseudo-effective or singular positively curved vector bundles (and more generally, torsion-free sheaves).

\section{Proof of the main result}\label{Sec3}
This section is devoted to the proof of Theorem \ref{thm-main}. 
We begin with the following elementary lemma: 

\begin{lemm}\label{lemm-ext}
Let $\mathcal{E}$ be a torsion-free sheaf on a compact complex manifold $X$. 
If $\mathcal{E}$ is pseudo-effective, then the reflexive hull $\Lambda^{[p]}\mathcal{E}$ 
of the $p$-th exterior product is also pseudo-effective. 
\end{lemm}
\begin{proof}
Fix a Hermitian form $\omega$ on $X$. 
Recall that the pseudo-effectivity of the torsion-free sheaf $\cal{E}$ is defined by the following condition:
For every $m \in \mathbb{Z}_{+}$, 
there exists a singular Hermitian metric $h_{m}$ on the $m$-th symmetric power 
$\Sym^{m} \cal{E}|_{X_{0}}$ such that $ \log |u|_{h_{m}^{*}}$ is $\omega$-psh 
 for any local section $u$ of the dual sheaf $(\Sym^{m} \cal{E})^{*}$. 

Consider an injective morphism of sheaves
$$
i \colon 
\Sym^{[m]}(\Lambda^{[p]}\mathcal{E}^{*}) \hookrightarrow \Sym^{[m]}(\Sym^{[p]}\mathcal{E}^{*}) =\Sym^{[pm]} \mathcal{E}^{*}
$$
for an integer $m \in \mathbb{Z}_{+}$. 
  Let $g$ be the singular Hermitian metric on $\Sym^{[m]}(\Lambda^{[p]}\mathcal{E})$ induced by $h:=h_{mp}$ and the above morphism. 
By construction, 
we see that $\log |v|_{g^{*}}=\log |i (v)|_{h^{*}}$ is $\omega$-psh 
for any local section $v$ of $\Sym^{[m]}(\Lambda^{[p]}\mathcal{E}^{*})$, 
finishing the proof. 
\end{proof}

We will prove Theorem \ref{thm-irr} after showing that a compact K\"ahler manifold with pseudo-effective tangent bundle 
is of special type in the sense of Campana.  
(see \cite[Definitions 1.19, 2.1, Proposition 1.25, Theorem 2.22]{Cam04} for the precise definition).

\begin{theo}\label{thm-sp}
Let $X$ be a compact K\"ahler manifold with pseudo-effective tangent bundle. 
Then $X$ is of special type in the sense of Campana, 
$($i.e.,\,there is no almost holomorphic dominant rational map $\phi \colon X \dashrightarrow Y$ of general type$).$
In particular, the image of any $\GL$-representation
$
\rho \colon \pi_{1}(X) \to \GL(r\mathord{:}\,\mathbb{C}) 
$ 
is virtually abelian. 
\end{theo}
\begin{proof}
Assume, for contradiction, that $X$ is not of special type in the sense of Campana. 
Then, by definition, there exists an almost holomorphic dominant rational map 
$\phi \colon X \dashrightarrow Y$ and a resolution $\bar \phi \colon \bar X \to Y$ of its indeterminacies, 
equipped with the following diagram:
\begin{equation*}
\xymatrix@C=40pt@R=30pt{
 & \bar X  \ar[d]_\pi \ar[rd]^{\bar{\phi}\ \ }  & \\ 
& X \ar@{-->}[r]^{\phi \ \ \ }  &  Y.\\   
}
\end{equation*}
such that
\[
\kappa(L) \,\geq\, m := \dim Y \;>\; 0,
\]
where 
\[
L \;:=\; \mathcal{O}_{\bar X}\bigl({\bar \phi}^{*} K_{Y}\bigr)_{\sat} \;\subset\; \Omega_{\bar X}^{m}
\]
is the saturation of the invertible subsheaf defined by pulling back the canonical bundle $K_{Y}$. 
Here $\kappa(\bullet)$ denotes the Kodaira dimension, and 
$\Omega_{\bar X}^{m}$ is the $m$-th exterior product of the cotangent bundle $\Omega_{\bar X}$
(see \cite[Definitions 1.19, 2.1, Proposition 1.25, Theorem 2.22]{Cam04}).

Taking the pushforward by $\pi$ and then the dual, 
we obtain a sheaf morphism
\[
\Lambda^{m} T_{X} \;\cong\; \bigl(\pi_{*}\,\Omega_{\bar X}^{m}\bigr)^{*} \;\longrightarrow\; \bigl(\pi_{*} L\bigr)^{*}.
\]
The isomorphism on the left-hand side follows from reflexivity. 
Lemma \ref{lemm-ext} shows that $\Lambda^{m} T_{X}$ is pseudo-effective. 
Since this morphism is generically surjective by construction, 
we deduce that $(\pi_{*} L)^{*}$ is a pseudo-effective line bundle, 
by noting that the quotient of pseudo-effective sheaves under generically surjective sheaf morphism is also 
pseudo-effective. 
Thus, any singular Hermitian metric $g$ on $(\pi_{*} L)^{**}$ with semi-positive curvature must be Hermitian flat. 

On the other hand, since $L$ is pseudo-effective, 
there exists a singular Hermitian metric $h$ on $L$ with semi-positive curvature. 
Let $X_{0} \subset X$ be a Zariski open set such that $\pi \colon \bar X \to X$ is an isomorphism over $X_{0}$ 
and $\codim(X \setminus X_{0}) \,\geq\, 2$. 
Consider the singular Hermitian metric on $(\pi_{*} L)^{**}\bigl|_{X_{0}} \cong \pi_{*} L$ induced by $h$, 
which extends to a singular Hermitian metric $g$ on $(\pi_{*} L)^{**}$ with semi-positive curvature. 
The above argument shows that $g$ must be Hermitian flat. 
By the construction of $g$, 
the support of $\sqrt{-1}\,\Theta_{h}$ is contained in the exceptional locus of $\pi \colon \bar X \to X$. 
Hence, by the support theorem of currents, we see that $\sqrt{-1}\,\Theta_{h}$ can be written as $\sqrt{-1}\,\Theta_{h}=[E]$, 
where $[E]$ is the integration current of an effective $\pi$-exceptional divisor $E$. 
This would imply $c_{1}(L) = c_{1}(E)$, contradicting $\kappa(L) > 0$. 

The latter conclusion follows from \cite[Theorem 7.8]{Cam04}.
\end{proof}

\begin{theo}\label{thm-irr}
Let $X$ be a compact K\"ahler manifold with pseudo-effective tangent bundle.
Then, the following statements hold$:$ 
\begin{itemize}
\item[$\rm{(1)}$] The inequality $\hat{q}(X) \leq \dim X$ holds, 
where $\hat{q}(X)$ denotes the augmented irregularity defined by 
$$
\hat{q}(X):=
\sup \{ 
q(X')\,|\, X' \to X \text{ is a finite \'etale cover} 
\}. 
$$
In particular, there exists a finite \'etale cover $X' \to X$ such that $q(X')=\hat{q}(X)$.

\item[$\rm{(2)}$]  If $\hat{q}(X) = 0$, then $X$ is a projective and rationally connected manifold. 
\end{itemize}
\end{theo}
\begin{proof}
(1) Let $\nu \colon X' \to X$ be a finite \'etale cover. 
Then $X'$ also has the pseudo-effective tangent bundle. 
Thus, by \cite[Theorem 3.12]{HIM22}, the Albanese map $X' \to \Alb(X')$ is surjective, 
showing 
\[
q(X') = \dim \Alb(X')  \leq \dim X' = \dim X,
\]
which proves the first statement.

\smallskip

(2) 
Once we know that $X$ is projective, by applying Theorem \ref{thm-main} (in the projective case), 
we can conclude that the base $Y$ of the MRC fibration of $X$ must be a single point by $\hat{q}(X)=0$, 
which shows that $X$ is rationally connected. 

Assume, for contradiction, that $X$ is not projective. 
It suffices to show that some finite \'etale cover of $X$ is projective. 
Since $X$ is non-projective, there exists a non-zero holomorphic $2$-form on $X$, 
inducing a non-zero morphism $\sigma \colon T_{X} \to \Omega_{X}$. 
Define the reflexive sheaf $\mathcal{F}$ by  the reflexive hull $\mathcal{F}:=(\sigma(T_{X}))^{**}$ of the image $\sigma(T_{X})$. 
Then, by construction, we have a generically surjective morphism
\[
\sigma \colon T_{X} \;\to\; \mathcal{F} \;\subset\; \Omega_{X}.
\]
Hence, the sheaf $\mathcal{F}$ is pseudo-effective. 
On the other hand, the dual sheaf $\mathcal{F}^{*}$ is also pseudo-effective 
since $\mathcal{F}^{*}$ is the quotient of a generically surjective morphism $T_{X} \to \mathcal{F}^{*}$. 
This implies that the line bundle 
$
\det \mathcal{F} := (\Lambda^{r} \mathcal{F})^{**}
$
has vanishing first Chern class, where $r$ is the rank of $\mathcal{F}$. 
In particular, the vector bundle 
\[
(\Lambda^{r}\Omega_{X} \otimes \det \mathcal{F}^{*})^{*} = \Lambda^{r} T_{X} \otimes \det \mathcal{F}
\]
is pseudo-effective. 
Thus, the induced sheaf morphism $\det \mathcal{F} \to \Lambda^{r}\Omega_{X}$ is injective as a bundle morphism 
by the proof of \cite[Lemma 2.9 (1)]{IMZ}. 
Note that the proof of \cite[Lemma 2.9 (1)]{IMZ} does not require the projectivity of $X$. 
Then, by \cite[Lemma 1.20]{DPS94}, the sheaf $\mathcal{F}$ is actually a subvector bundle of $\Omega_{X}$. 
Note that a pseudo-effective vector bundle in our paper is strongly pseudo-effective in the sense of  \cite[Definition 1]{Wu22}. 
Thus, by applying \cite[Main Theorem, Corollary]{Wu22} to $\mathcal{F}$, 
we conclude that $\mathcal{F}$ is a numerically flat locally free sheaf (cf.\,\cite[Theorem 2.11]{MWb}). 
In particular $\mathcal{F}$ is a flat vector bundle.

Consider the $\GL$-representation associated with $\mathcal{F}$: 
\[
\rho \colon \pi_{1}(X) \to \GL(r\mathord{:}\,\mathbb{C}).
\] 
The latter conclusion of Theorem \ref{thm-sp} shows that $\Image(\rho)$ is virtually abelian. 
Thus, after possibly replacing $X$ with a finite \'etale cover, 
we may assume $\Image(\rho)$ is abelian. 
The first homology group $H_{1}(X, \mathbb{Z})$ is the abelianization of $\pi_{1}(X)$. 
By the universal property of abelianization, the representation $\rho$ factors through 
\[
\pi_{1}(X) \;\to\; H_{1}(X, \mathbb{Z}) \;\to\; \GL(r\mathord{:}\,\mathbb{C}).
\]
On the other hand, since $q(X) = 0$ and $X$ is K\"ahler, 
we see that $H_{1}(X,\mathbb{Z})$ is a torsion group (in particular, it is a finite group). 
Hence, the image $\Image(\rho)$ is finite, 
which implies there exists a finite \'etale cover $\nu \colon X' \to X$ such that 
$\nu^{*}\mathcal{F} \subset \nu^{*}\Omega_{X} = \Omega_{X'}$ is a trivial vector bundle on $X'$. 
Then $X'$ admits a non-zero $1$-form, contradicting $\hat{q}(X)=q(X')=0$.
\end{proof}

We now complete the proof of Theorem \ref{thm-main}.

\begin{proof}[Proof of Theorem \ref{thm-main}]
First, we prove the statement for a compact K\"ahler manifold $X$ with a pseudo-effective tangent bundle 
by induction on $\dim X$. 
Suppose that the statement is known in dimensions less than $n := \dim X$.
It is sufficient to prove the conclusion 
after we replace $X$ with a finite \'etale cover 
by the argument in \cite{CH19} and \cite[Corollary 2.11]{Hor07} 
(see also the proof of \cite[Theorem 1.1]{Mat}). 

Consider the Albanese map $\phi \colon X \to Y := \Alb(X)$. 
By replacing $X$ with a finite \'etale cover, we may assume $\hat{q}(X)=q(X)=\dim Y$ by Theorem \ref{thm-irr} (1). 
Note that $\phi \colon X \to Y$ is a smooth fibration, $Y$ is a finite \'etale quotient of a torus, 
and, a very general fiber $F$ of $\phi$ also has the pseudo-effective tangent bundle (see \cite[Theorem 3.12]{HIM22}). 
By Theorem \ref{thm-irr} (2), it suffices to prove that $\hat{q}(F)=0$.

Replacing $X$ with the fiber product induced by a finite \'etale cover $T \to Y$, where $T$ is a torus, 
we may assume $Y$ is  a torus. 
By the induction hypothesis, we may assume that $F$ admits a fibration satisfying the properties in Theorem \ref{thm-main}, 
in particular, we see that $\pi_{1}(F)$ is virtually abelian. 
Then, we can apply \cite[Proposition 3.12, Remark 3.13]{DPS94} to obtain  
\[
\dim Y=\hat{q}(X) = \hat{q}(F) + \hat{q}(Y) = \hat{q}(F) + \dim Y,
\]
which shows $\hat{q}(F) = 0$.

\smallskip

Finally, we assume $T_X$ admits a positively curved singular Hermitian metric, 
and then will prove that $\phi \colon X \to Y$ is a locally constant fibration.  
By \cite[Theorem 3.12]{HIM22}, the standard exact sequence 
\[
0 \;\longrightarrow\; T_{X/Y} 
\;\longrightarrow\; T_X \;\longrightarrow\; \phi^* T_Y \;\longrightarrow\; 0 
\]
splits, and thus $\phi \colon X \to Y$ is locally trivial by Ehresmann's theorem (see also \cite[Lemma 3.19]{Hor07}). 
Moreover, since $\phi \colon X \to Y$ is an MRC fibration of $X$, 
it is actually a projective morphism by \cite[1.1 Theorem]{CH}. 
Therefore, by \cite[Lemma 3.1]{Mul}, we can conclude that $\phi \colon X \to Y$ is a locally constant fibration. 
Strictly speaking, \cite[Lemma 3.1]{Mul} assumes $Y$ is projective, 
but this assumption can be removed by checking the proof in \cite{Mul}. 
Indeed, the proof there is divided into three steps: 
the first two steps do not require $Y$ to be projective (only that $\phi$ itself is a projective morpshim), 
while the third step uses \cite[Theorem 1.6 and Proposition 1.7]{Mul} 
to infer numerical flatness from the existence of holomorphic connections. 
Although projectivity appears to be used, 
in practice \cite[Remark 3.7.(ii)]{Bis95} and Simpson's correspondence hold for compact K\"ahler manifolds, 
which ensures \cite[Theorem 1.6 and Proposition 1.7]{Mul} remain valid 
for compact K\"ahler manifolds as well.
\end{proof}


\end{document}